\documentclass[11pt,leftnumbering]{amsart}
\usepackage{graphicx}
\usepackage{psfrag}
\usepackage{latexsym}
\usepackage{amssymb}
\usepackage{amsfonts}
\usepackage{amsmath}
\usepackage{amsthm}
\usepackage{color}
\usepackage[margin=1in]{geometry}

\newcommand\R{{\mathbb R}}

\newcommand{\E}{\mathbb{E}}

\newtheorem{theorem}{Theorem}[section]

\newtheorem{lemma}[theorem]{Lemma}
\newtheorem{proposition}[theorem]{Proposition}

\theoremstyle{definition}

\theoremstyle{remark}
\newtheorem{remark}[theorem]{Remark}
\newtheorem{question}[theorem]{Question}

\numberwithin{theorem}{section}
\numberwithin{equation}{section}

\newcommand{\ep}{\epsilon}
\newcommand{\del}{\delta}

\newcommand{\reals}{\mathbb{R}}
\newcommand{\Heis}{\mathbb{H}}

\newcommand{\nats}{\mathbb{N}}

\newcommand{\nbhd}{\mathcal{N}}

\newcommand{\inv}{^{-1}}

\newcommand{\ovl}{\overline}

\newcommand{\til}{\widetilde}

\newcommand{\cH}{\mathcal{H}}
\newcommand{\subeq}{\subseteq}

\newcommand{\bslash}{\backslash}

\newcommand{\Lloc}{\rm{L}_{\loc}}

\newcommand{\V}{\mathbb{V}}

\def\loc{\operatorname{loc}}
\def\diam{\operatorname{diam}}
\def\card{\operatorname{card}}

\def\Lip{\operatorname{Lip}}

\def\W{\operatorname{W}}
\def\Sl{\operatorname{Sl}}

\def\XXint#1#2#3{{\setbox0=\hbox{$#1{#2#3}{\int}$}
\vcenter{\hbox{$#2#3$}}\kern-.5\wd0}}

\newcommand{\restrict}{\begin{picture}(12,12)
                        \put(2,0){\line(1,0){8}}
                        \put(2,0){\line(0,1){8}}
                       \end{picture}}

\author[Z.M. Balogh, J.T. Tyson, and K. Wildrick]{Zolt\'an M. Balogh, Jeremy T. Tyson, Kevin Wildrick}

\address{Z. M. Balogh: Mathematisches Institut, Universit\"at Bern, Sidlerstrasse 5, 3012 Bern, Switzerland ({\tt balogh.zoltan@math.unibe.ch})}

\address{J. T. Tyson: Department of Mathematics, University of Illinois at Urbana-Champaign, 1409 W Green Street, Urbana, IL 61801, USA ({\tt tyson@math.uiuc.edu})}

\address{K. Wildrick: Mathematisches Institut, Universit\"at Bern, Sidlerstrasse 5, 3012 Bern, Switzerland ({\tt kevin.wildrick@math.unibe.ch})}


\thanks{The first and third authors were supported by the Swiss National Science Foundation, European Research Council Project CG-DICE, and the European Science Council Project HCAA. The second author was supported by NSF grants DMS-0901620 and DMS-1201875.}

\begin{document}
\title[Frequency of Sobolev dimension distortion in Heisenberg groups]{Frequency of Sobolev dimension distortion of horizontal subgroups of Heisenberg groups}
\begin{abstract}
We study the behavior of Sobolev mappings defined on the Heisenberg groups with respect to a foliation by left cosets of a horizontal homogeneous subgroup.  We quantitatively estimate, in terms of Euclidean Hausdorff dimension, the size of the set of cosets that are mapped onto sets of high dimension.  The proof of our main result combines ideas of Gehring and Mostow about the absolute continuity of quasiconformal mappings with Mattila's projection and slicing machinery.
\end{abstract}
\date{}

\maketitle

\begin{center}
\it In memory of Frederick W. Gehring (1925--2012)
\end{center}

\

\section{Introduction}

Every Sobolev mapping $f$ defined on an open subset of Euclidean space admits a representative which is absolutely continuous along almost every line. In particular, the image of almost every line segment has dimension no greater than~$1$. While $f$ may increase the Hausdorff dimension of a line segment in the remaining measure zero set of lines, if $f$ is super-critical (i.e., $f \in \W^{1,p}_{\loc}$, $p>n$), then the amount of this increase is controlled. The following folklore theorem states that there is a universal bound on the amount by which such a mapping $f$ can increase the dimension of a subset. Here and below, we denote by $\cH^\alpha_X$ the $\alpha$-dimensional Hausdorff measure in a metric space $X$, and by $\dim_X(A)$ the Hausdorff dimension of a subset $A$ of the space $X$; when it will not cause confusion, we may suppress the reference to the ambient space $X$.

\begin{theorem}\label{Kaufman} Let $f \colon \reals^n \to Y$ be a continuous mapping to a metric space $Y$ that lies in the Sobolev space $\W^{1,p}_{\loc}(\reals^n;Y)$, $p > n$. For any subset $E \subeq \reals^n$ that is of $\sigma$-finite $s$-dimensional Hausdorff measure for some $0\leq s<n$, it holds that $\cH^{\alpha}(f(E))=0$, where
\begin{equation}\label{Kaufman-alpha}
\alpha = \frac{ps}{p-(n-s)}.
\end{equation}
Moreover, if $\cH_{\reals^n}^n(E)=0$, then $\cH^n(f(E))=0$ as well.
\end{theorem}

A proof may be found, for example, in \cite[Theorem 1]{Kaufman}. This bound is strictly better than that provided by the $(1-\frac{n}{p})$-H\"older continuity of such mappings arising from the Sobolev Embedding Theorem, and, as shown by Kaufman in \cite[Theorem 2]{Kaufman}, it is sharp. Estimates such as those that appear in Theorem \ref{Kaufman} have been known for many years in the quasiconformal category, see e.g.\ Gehring--V\"ais\"al\"a \cite{GehringV73} and Astala \cite{Astala94}.

It is also possible to bound the ``quantity" of sets of a given dimension that can be simultaneously distorted by a Sobolev mapping. One way to formulate this statement precisely is in terms of foliations. Balogh, Monti, and Tyson \cite{BMT} considered the foliation of $\reals^n$ by translates of a linear subspace $V$ and estimated the size of the set of translates, as measured by Hausdorff dimension in $V^\perp$, that are mapped by a super-critical Sobolev mapping onto sets of pre-specified Hausdorff dimension between $m = \dim V$ and the universal bound given by \eqref{Kaufman-alpha}.

\begin{theorem}[Balogh--Monti--Tyson]\label{BMT theorem} Let $f$ be a continuous mapping in the Sobolev space $\W^{1,p}_{\loc}(\reals^n;Y)$, $p > n$. Given a vector subspace $V$ of $\reals^n$ of dimension $1 \leq m \leq n$, and given
$$ \alpha \in \left\lbrack m, \frac{pm}{p-(n-m)}\right\rbrack,$$
it holds that
\begin{equation}\label{BMT setup} \dim_{\reals^n} \{a \in V^\perp: \dim{f(a+V)} \geq \alpha\} \leq \beta(p,m,\alpha),\end{equation}
where
\begin{equation}\label{BMT formula}  \beta(p,m,\alpha)= (n-m)-p\left(1-\frac{m}{\alpha}\right).\end{equation}
\end{theorem}

In fact, a slightly different result was stated in \cite{BMT}: excluding the endpoint $\alpha = m$, the authors conclude that
$$\cH_{\reals^{n}}^\beta\left(\{a \in V^\perp: \dim{f(a+V)} > \alpha\} \right) =0,$$
which recovers the universal estimate given in Theorem \ref{Kaufman} as a special case.  Essentially the same proof yields Theorem \ref{BMT theorem} as stated above. Theorem \ref{BMT theorem} is also sharp as is demonstrated by examples in \cite{BMT}.

For extensions of Theorem \ref{BMT theorem} to the sub-critical ($p<n$) case, see Hencl--Honz{\'\i}k \cite{hh:subspaces}, and for generalizations and additional results for planar quasiconformal maps, see Bishop--Hakobyan \cite{bh:freq}.

Recent advances in analysis on metric spaces motivate extensions of the Sobolev theory to non-Euclidean source spaces. One of the most ubiquitous and important examples of such a source space is the Heisenberg group $\Heis^n$, a non-commutative Lie group equipped with a left-invariant metric arising from a non-integrable tangent plane distribution on $\reals^{2n+1}$, called the \emph{horizontal distribution}. Although $\Heis^n$ is Ahlfors $(2n+2)$-regular and supports a $1$-Poincar\'e inequality, conditions which ensure that a large portion of first-order Euclidean analysis remains valid, its fractal geometry differs drastically from that of Euclidean space.

The theory of Sobolev mappings on the Heisenberg groups arises naturally from the study of partial differential equations and quasiconformal mappings. Quasiconformal mappings in the Heisenberg groups feature prominently in Mostow's celebrated rigidity theorem for rank one symmetric spaces \cite{MostowBook}, \cite{MostowRemark}. Of particular importance is the fact that such mappings are absolutely continuous on almost every line tangent to the horizontal distribution, which leads to needed differentiability results. In essence, this result allows one to consider quasiconformal mappings on the Heisenberg group from both geometric and analytic perspectives. For an excellent overview of the theory of Heisenberg quasiconformal maps, see Kor\'anyi and Reimann \cite{KR}.

The aim of this paper is to provide an analog of Theorem \ref{BMT theorem} for Sobolev mappings defined on the Heisenberg groups $\Heis^n$. Our new results are Theorems \ref{main} and \ref{four corners intro}.  Relevant background information for this paper on Sobolev mappings defined on the Heisenberg groups is given in Section~\ref{notation section}.

We consider foliations of $\Heis^n$ by left cosets of arbitrary homogeneous subgroups that are tangent to the horizontal distribution. Such a subgroup $\V$ is called \emph{horizontal} and may be identified with an isotropic subspace $V$ of $\reals^{2n}$; its dimension, now denoted by $m$, satisfies $1 \leq m \leq n$. The set of leaves of this foliation is parameterized by the \emph{vertical complement} $\V^\perp=V^\perp \times \reals$, where $V^\perp$ is the Euclidean orthogonal complement of $V$ in $\reals^{2n}$ and the additional copy of $\reals$ corresponds to the $t$-axis in $\Heis^n$. This yields the semidirect decomposition $\Heis^n = \V^\perp \ltimes \V$, and allows us to define a mapping
\begin{equation}\label{vert proj}
\pi_{{\V}^\perp} \colon \Heis^n \to {\V}^\perp.
\end{equation}
The preimage of a point $a \in \V^\perp$ is the left coset $a*\V$, i.e., a leaf of the foliation under consideration. In case $m=1$ such leaves are precisely the integral curves of a horizontal vector field defining the one-dimensional horizontal subspace $\V$, which are used in the standard definition of the ACL property on the Heisenberg group \cite{KR}. These concepts are discussed further in Section~\ref{notation section} below.

To begin, the universal bound on dimension distortion by super-critical Sobolev mappings remains valid in the Heisenberg groups, as was shown in our previous work \cite[Theorem 4.1]{DimDistMetric}.

\begin{theorem}\label{Kaufman Heis}
Let $f$ be a continuous mapping in the Sobolev space $\W^{1,p}_{\loc}(\Heis^n;Y)$, $p > 2n+2$. For any subset $E \subeq \Heis^n$ with $\sigma$-finite $\cH_{\Heis^n}^{s}$ measure, it holds that $\cH^{\alpha}(f(E))=0$, where
\begin{equation}\label{Kaufman Heis alpha}
\alpha = \frac{ps}{p-(2n+2-s)}.
\end{equation}
Moreover, if $\cH^{2n+2}_{\Heis^n}(E)=0$, then $\cH^n(f(E))=0$ as well.
\end{theorem}
Theorem \ref{Kaufman Heis} is sharp. Indeed, the set of such mappings that distort a given subset by the maximal amount is prevalent \cite[Theorems~1.3 and 1.4]{DimDistMetric}.

The fact that Euclidean projection onto a subspace is Lipschitz plays a major role in Gehring's proof of the ACL property of Euclidean quasiconformal mappings and in the proof of Theorem \ref{BMT theorem}.   However, in Heisenberg group, the projection $\pi_{\V^\perp}$ defined above is \emph{not} Lipschitz on compact sets. This complication in the proof of the ACL property of quasiconformal maps on the Heisenberg group and other Carnot groups played a key role in motivating the work of Heinonen and Koskela that inaugurated the modern theory of analysis on metric spaces \cite{Definitions}, \cite{Acta}. We review this complication and Mostow's amended proof for the ACL property of Heisenberg quasiconformal maps in section \ref{discussion section}. Our main results rely heavily on ideas drawn from both Gehring's original argument in the Euclidean setting and Mostow's amended proof in the Heisenberg setting.

A key point in this work is that although $\pi_{\V^\perp}$ fails to be Lipschitz on compact sets when the target is metrized as a subset of $\Heis^n$, it is Lipschitz on compact sets when $\V^\perp$ is metrized as a subset of the \emph{Euclidean} space $\reals^{2n+1}$. Our main results is as follows: given a target dimension $\alpha$ between $m$ and the universal upper bound given in \eqref{Kaufman Heis alpha}, we quantitatively estimate, in terms of Hausdorff dimension in $\V^\perp$ \emph{equipped with the Euclidean metric from $\reals^{2n+1}$}, the size of the set of left cosets of $\V$ that are mapped onto a set of dimension at least $\alpha$ by any super-critical Sobolev mapping. We note that
$$\dim_{\Heis^n} \V^\perp = 2n+2-m,$$
while
$$\dim_{\reals^{2n+1}} \V^\perp= 2n+1-m.$$
Using the Euclidean Hausdorff dimension to measure of the size of the set of left-cosets of a horizontal supgroup that are distorted by a Sobolev mapping allows us to state the main result of the paper. See Figure \ref{plot1}.
\begin{theorem}\label{main}
Let $f$ be a continuous mapping in the Sobolev space $\W^{1,p}_{\loc}(\Heis^n;Y)$, $p > 2n+2$. Given a horizontal subgroup $\V$ of $\Heis^n$ of dimension $1 \leq m \leq n$, and
$$\alpha \in \left[m,\frac{pm}{p-(2n+2-m)}\right],$$
it holds that
$$\dim_{\reals^{2n+1}}\{a \in \V^\perp:  \dim f(a*\V) \geq \alpha\} \leq \beta(p,m,\alpha),$$
where
\begin{equation}\label{beta definition}
\beta(p,m,\alpha) = \begin{cases}
(2n+1-m) - \frac{p}{2}\left(1-\frac{m}{\alpha}\right)& \alpha \in \left[m,\frac{pm}{p-2}\right], \\
(2n+2-m) - p\left(1-\frac{m}{\alpha}\right) & \alpha \in \left[\frac{pm}{p-2},\frac{pm}{p-(2n+2-m)}\right]. \\
\end{cases}\end{equation}
\end{theorem}

\begin{figure}[h]
\begin{center}
\input{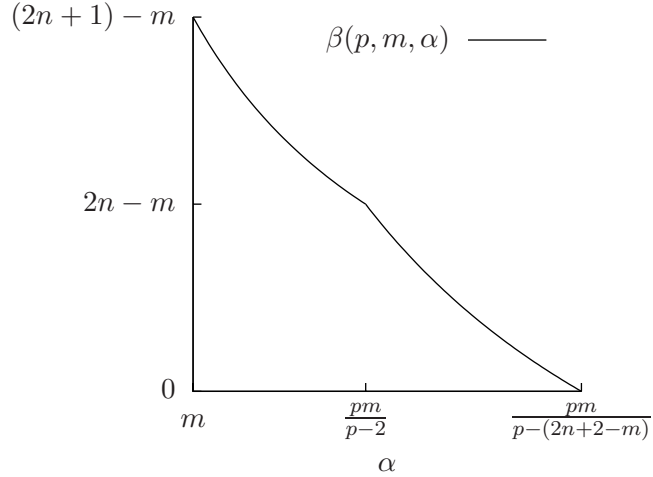}
\caption{The quantity $\beta(p,m,\alpha)$ of Theorem \ref{main} as a function of $\alpha$.}\label{plot1}
\end{center}
\end{figure}
	
In the special case of the foliation of $\Heis = \Heis^1$ by left cosets of the $x$-axis $\V_x$, Theorem \ref{main} yields
$$\dim_{\reals^3}\{a=(0,y,t) \in \Heis:  \dim f(a*\V_x) \geq \alpha\} \leq \begin{cases}
2- \frac{p}{2}\left(1-\frac{1}{\alpha}\right) & \alpha \in \left[1,\frac{p}{p-2}\right],\\
3 - p\left(1-\frac{1}{\alpha}\right) & \alpha \in \left[\frac{p}{p-2},\frac{p}{p-3}\right]. \\ \end{cases}$$

The bifurcated nature of the conclusion of Theorem \ref{main} is typical for the Heisenberg groups; in this case, the bifurcation occurs when
$$\beta\left(p,m,\frac{pm}{p-2}\right) = 2n-m = \dim_{\reals^{2n}}V^\perp.$$

Theorem \ref{main} provides the expected estimates at the endpoints of the interval of definition of $\beta$ as a function of $\alpha$. The estimate when $\alpha=m$ is trivial to verify, and taking $f \colon \Heis^n \to \Heis^n$ to be the identity mapping shows that it cannot be improved. On the other hand, when $\alpha$ is the universal upper bound provided by Theorem \ref{Kaufman Heis}, we have
$$\beta\left(p,m,\frac{pm}{p-((2n+2)-m)}\right)=0,$$
as expected. However, we do not recover Theorem \ref{Kaufman Heis} as a special case as the conclusion of Theorem~\ref{main} does not assert that the exceptional set $\{a \in \V^\perp:  \dim f(a*\V) > \alpha\}$ has $\cH_{\reals^{2n+1}}^\beta$ measure zero. This issue is discussed in greater detail in Question \ref{recover} in Section \ref{questions section} below.

The estimate
$$\dim_{\reals^{2n+1}}\{a \in \V^\perp:  \dim f(a*\V) \geq \alpha\} \leq \beta(p,m,\alpha)$$
for the Euclidean Hausdorff dimension of the exceptional set can be converted into an estimate of its Heisenberg Hausdorff dimension:
$$\dim_{\Heis^n}\{a \in \V^\perp:  \dim f(a*\V) \geq \alpha\} \leq \beta_{\Heis^n}(p,m,\alpha).$$
The value of $\beta_{\Heis^n}$ can be computed from $\beta$ using the Dimension Comparison Theorem of \cite{Comparison} and \cite{Comp}. The resulting estimate is unlikely to be sharp and is not very aesthetic, and so we omit it. A conjectured optimal value for $\beta_{\Heis^n}$ is provided in Question \ref{ideal} in Section \ref{questions section} below.

The sharpness of Theorem \ref{main} throughout the interval of definition of $\beta$ is unclear.  However, we are able to construct examples of Sobolev mappings that distort the dimension of a large set of leaves by a small amount. Our construction is based on a similar example given in \cite{BMT}.

\begin{theorem}\label{four corners intro} Let $\V_x$ denote the horizontal subgroup defined by the $x$-axis in $\Heis$, and let $p>4$. For each
$$\alpha \in \left(1,\frac{p}{p-2}\right)$$ there is a compact set $E_\alpha \subeq \V_x^\perp$ and a continuous mapping $f \in \W^{1,p}(\Heis;\reals^2)$ such that
$$0<\cH_{\reals^3}^{2-p\left(1-\frac{1}{\alpha}\right)}(E_\alpha) < \infty$$
and $\dim f(a*\V) \geq \alpha$ for every $a \in E_\alpha$.
\end{theorem}

\begin{figure}[h]
\begin{center}
\input{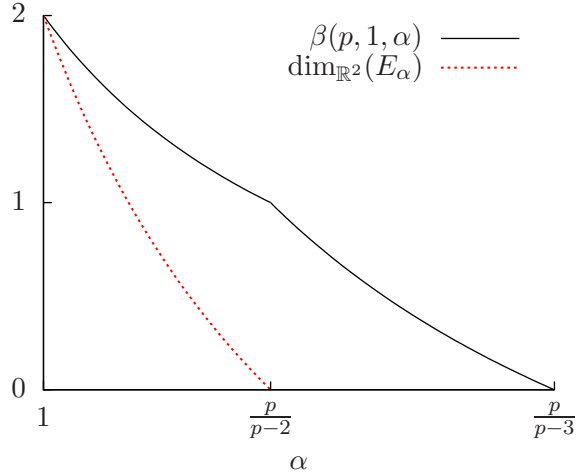}
\caption{The quantity $\beta(p,1,\alpha)$ of Theorem \ref{main} and the dimension of the set $E_\alpha$ of Theorem \ref{four corners intro}.}\label{plot2}
\end{center}
\end{figure}

We note that the dimension of $E$ given above is strictly smaller that the corresponding estimate given by $\beta(p,1,\alpha)$ in Theorem \ref{main}. In particular, the dimension of $E$ tends to $0$ when $\alpha$ tends to $\frac{p}{p-2}$, whereas the estimate given by Theorem \ref{main} in this setting is
$$\beta\left(p,1,\frac{p}{p-2}\right)=1.$$ See Figure \ref{plot2}.

A first attempt at producing a theorem similar to Theorem~\ref{main} was made by adapting the Euclidean proof of Theorem \ref{BMT theorem} to the setting of a metric measure space $X$ that supports a Poincar\'e inequality and is equipped with a locally David--Semmes $s$-regular foliation $\pi\colon X \to W$, $s \geq 0$.  Such foliations, which were studied by David and Semmes as analogs of projections \cite{RegularBetween}, are locally Lipschitz mappings having the property that the localized preimage of a ball of radius $r$ in $W$ can be covered by approximately $r^{-s}$ balls of radius $r$ in $X$. We gave the following theorem in \cite{DimDistMetric}; we refer to that paper for the relevant definitions and background.

\begin{theorem}\label{foliation intro} Let $Q\geq 1$ and $0<s<Q<p$.  Let  $(X,d_X,\mu)$ be a metric measure space that is proper, locally homogeneous of dimension at most $Q$, supports a local $Q$-Poincar\'e inequality, and is equipped with a locally David--Semmes $s$-regular foliation $(X,W,\pi)$. Let $Y$ be an arbitrary metric space and $f \colon X \to Y$ a continuous mapping with an upper gradient in $\Lloc^p(X)$. For $\alpha \in \left( s,\frac{ps}{p-Q+s}\right\rbrack$, it holds that
$$\dim_W \{a \in W: \dim f(\pi\inv(a))\geq \alpha\} \leq (Q-s)-p\left(1-\frac{s}{\alpha}\right).$$
\end{theorem}

Foliations of the Heisenberg group by \emph{right} cosets of horizontal or vertical subgroups fit nicely in to the framework of Theorem \ref{foliation intro}, as described in \cite{DimDistMetric}.  However, in the case of foliations by \emph{left} cosets of horizontal subgroups, Theorem \ref{foliation intro} is deficient.  It follows from the definitions that a leaf of a locally David--Semmes $s$-regular foliation, i.e., the preimage in $X$ of a point in $W$, has Hausdorff dimension no greater than $s$. However, the dimension of a leaf may be strictly smaller than $s$, a fact which makes Theorem \ref{foliation intro} inappropriate for the foliations by left cosets of horizontal subgroups as considered in this paper.

To wit, let $\V$ be a horizontal subgroup of $\Heis^n$ of dimension $m$. Although the projection mapping $\pi_{\V^\perp}$  is Lipschitz on compact sets when considered as a map into $\reals^{2n+1}$, it defines a $(m+1)$-foliation and not a $m$-foliation \cite[Section~6.3]{DimDistMetric}. Applying Theorem \ref{foliation intro} to this setting for
$$
\alpha \in \left(m+1, \frac{p(m+1)}{p-((2n+1)-m)} \right\rbrack,
$$
would result in replacing $\beta(p,m,\alpha)$ in \eqref{beta definition} by
\begin{equation}\label{Heis not Lipschitz estimate}
\beta(p,m,\alpha)=((2n+1)-m)-p\left(1-\frac{m+1}{\alpha}\right),
\end{equation}
which does not confirm the universal bound from Theorem \ref{Kaufman Heis}, and provides no information when $\alpha$ is in the range $[m,m+1)$. In fact, Theorem \ref{main} provides the only known estimates for the size of the exceptional set of left cosets of horizontal $m$-dimensional subspaces, valid for $\alpha$ near $m$, which provides a genuine reduction in dimension below the Hausdorff dimension of the complementary vertical subspace. See Figure \ref{plot3}.

\begin{figure}[h]
\begin{center}
\input{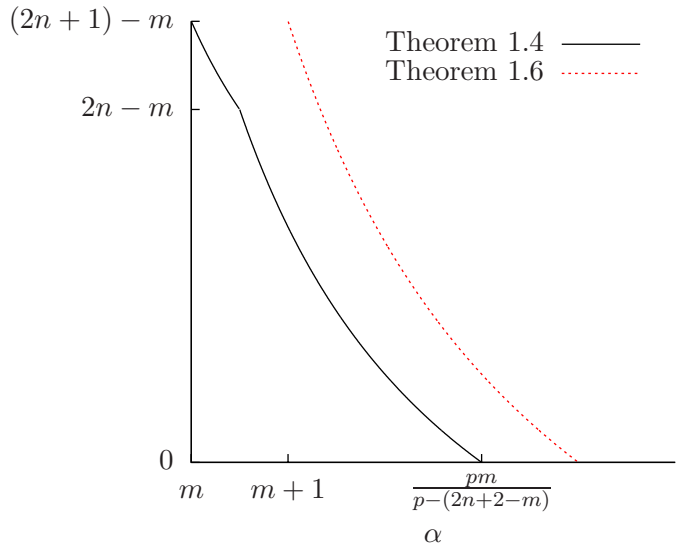}
\caption{The quantity $\beta(p,m,\alpha)$ of Theorem \ref{main} versus that of Theorem \ref{foliation intro}, shown here when $n=2$, $m=1$, and $p=6$.}\label{plot3}
\end{center}
\end{figure}

\

\paragraph{\bf Overview.}
In Section \ref{notation section}, we establish notation and conventions for the Heisenberg groups and their homogeneous subgroups.  Section \ref{discussion section} contains the proof of Theorem \ref{main}. Outlines of Gehring's and Mostow's proofs of the ACL property of quasiconformal mappings are included to clarify the structure of our argument. Section \ref{4 corners} contains the example asserted in Theorem \ref{four corners intro}. Section \ref{questions section} contains some questions left open by this paper.

\

\paragraph{\bf Acknowledgements.}
Research for this paper was carried out during a visit of JTT to the University of Bern in Summer 2012 and a visit of KW to the University of Illinois in February 2013. The hospitality of both institutions is appreciated.

\section{Notation and properties of the Heisenberg group}\label{notation section}
We employ standard notation for metric spaces. Given a metric space $(X,d)$, a point $x \in X$ and a radius $r>0$, we denote
$$B_X(x,r) = \{y \in X: d(x,y)<r\} \ \text{and}\ \ovl{B}_{X}(x,r) = \{y \in X: d(x,y)\leq r\}.$$
The open $r$-neighborhood of a set $A \subeq X$ is denoted
$$\nbhd_{X}(A,r) = \bigcup_{x \in A} B_{X}(x,r).$$
Where it will not cause confusion, we will suppress reference to the ambient space $(X,d)$, and a similar convention will hold for all quantities that depend implicitly on $(X,d)$.

We write $A\lesssim B$, resp.\ $A\gtrsim B$ to indicate that the inequality $A\le CB$, resp.\ $B\le CA$ holds, where $C$ is a constant depending only on suitable data (which will be indicated in practice or clear from context). We write $A\simeq B$ if $A\lesssim B$ and $B\lesssim A$.

\subsection{Basic properties and notation}
The Heisenberg group $\Heis^n$, $n \in \nats$, is the unique step two nilpotent stratified Lie group with topological dimension $2n+1$ and one dimensional center. We denote $\Heis^1=\Heis$.  As a set, we identify $\Heis^n$ with $\reals^{2n+1}$ equipped with coordinate system $(x_1,y_1,\hdots,x_n,y_n,t)$, which we also denote by $(z,t)$.
Given points $a=(z,t)$ and $a_0=(z',t')$, the group law on $\Heis^n$ is defined by
$$ a*{a'} =(z+z',t+t'+ 2\omega(z,z'))$$
where $\omega(z,z') = \sum_{i=1}^n (x_iy'_i - x'_iy_i)$ is the standard symplectic form on $\reals^{2n}$. The group $\Heis^n$ is equipped with a left-invariant metric $d_{\Heis^n}(a,a') = ||a^{-1}*a'||_{\Heis^n}$ via the \emph{Kor\'anyi norm}
$$
||a||_{\Heis^n} =( ||z||_{\reals^{2n}}^4 + |t|^2)^{1/4}.
$$
The metric space $(\Heis^n,d_{\Heis^n})$ is proper and Ahlfors $(2n+2)$-regular when equipped with its Haar measure (which agrees up to constants with both the Lebesgue measure in the underlying Euclidean space $\R^{2n+1}$ and the $(2n+2)$-dimensional Hausdorff measure $\cH^{2n+2}_{\Heis^n}$ in the Kor\'anyi metric $d_{\Heis^n}$). It is known that the metric measure space $(\Heis^n,d_{\Heis^n},\cH^{2n+2}_{\Heis^n})$ supports a $p$-Poincar\'e inequality for every $1\le p<\infty$; see \cite[Chapter 11]{SobMet} and the references therein.

The Heisenberg group $\Heis^n$ admits a one-parameter family of \emph{intrinsic dilations} $\{\del_r \colon \Heis^n \to \Heis^n\}_{r>0}$ defined, for a point $a=(z,t) \in \Heis^n$, by
$$\del_r(a)=(rz,r^2t).$$ These dilations commute with the group law and are homogeneous of order one with respect to the Kor\'anyi norm, i.e.,
$$\del_r(a)*\del_r(a')=\del_r(a*a') \ \text{and}\ ||\del_r(a)||_{\Heis^n} = r ||a||_{\Heis^n}.$$

\subsection{Homogeneous subgroups of $\Heis^n$}
A subgroup of $\Heis^n$ is \emph{homogeneous} if it is invariant under intrinsic dilations. Homogeneous subgroups come in two types. A homogeneous subgroup is called \emph{horizontal} if it is of the form $V \times\{0\}$ for an isotropic subspace $V$ of the symplectic space $\reals^{2n}$; recall that $V$ is {\it isotropic} if $\omega|_V=0$. Every homogeneous subgroup that is not horizontal contains the $t$-axis; these subgroups are called \emph{vertical}. Any horizontal subgroup $\V = V \times \{0\}$ defines a semidirect decomposition $\Heis^n = \V^\perp \ltimes \V$ where $\V^\perp = V^\perp \times \reals$ is the \emph{vertical complement} of $\V$; here $V^\perp$ denotes the usual orthogonal complement of $V$ in $\reals^{2n}$.

Since $\omega$ vanishes on isotropic subgroups, the restriction of the Kor\'anyi metric to horizontal subgroups coincides with the Euclidean metric. Consequently,
$$
\dim_{\Heis^n} \V = \dim_{\R^{2n+1}} \V = \dim V
$$
for each horizontal homogeneous subgroup; we write $\dim \V$ without any subscript in this case.  In the remainder of the paper, we will be working with a fixed horizontal homogeneous subgroup $\V$, and so \textbf{unless otherwise noted, we will denote $\dim \V$ by the letter $m$.}

On the other hand, the Heisenberg metric on the vertical complement $\V^\perp$ differs drastically from the Euclidean metric on $(V^\perp \times \reals) \subeq \reals^{2n+1}$; recall that the $t$-axis has Hausdorff dimension $2$ in the Heisenberg metric.  Since the integer $n \geq 1$ will be fixed throughout, given a subset $A \subeq \V^\perp$, we write
$$\dim_{\Heis} A := \dim_{\Heis^n}A, \ \text{and}\  \dim_{\reals} A := \dim_{\reals^{2n+1}} A.$$
A similar convention will be used for all notions that depend on the choice of Euclidean or Heisenberg metric. \textbf{Unless otherwise noted, we will denote $\dim_{\reals}\V^\perp$ by $w$}, so that
$$\dim_{\reals} \V^\perp = w=(2n+1)-m, \ \text{and}\ \dim_{\Heis} \V^\perp = w+1 = (2n+2)-m.$$
As mentioned in the introduction, the semidirect product decomposition $\Heis^n = \V^\perp \ltimes \V$ defines maps
\begin{align*} \pi_{\V} \colon  \Heis^n \to \V \ \text{and}\ \pi_{{\V}^\perp} \colon \Heis^n \to {\V}^\perp \end{align*}
by the formulas $\pi_{\V}(a) = a_{\V}$ and $\pi_{\V^\perp}(a) = a_{\V^\perp}$, where $a =a_{\V^\perp} * a_{\V}$.  The map $\pi_{{\V}^\perp}$ is not Lipschitz on compact sets when $\V^\perp$ is equipped with the Heisenberg metric, but it is Lipschitz on compact sets when $\V^\perp$ is equipped with the Euclidean metric inherited from $\reals^{2n+1}$. The proof of this fact and further information about the metric and measure-theoretic properties of these projection maps can be found in \cite{balogh-faessler-mattila-tyson2012} and \cite[Section~6.3]{DimDistMetric}.

 \subsection{Sobolev mappings on $\Heis^n$}
All results stated in this paper are given for globally defined mappings only for convenience. The methods in use are local in nature and pass without difficulty to mappings defined on open subsets of the ambient space.

We now discuss the definition of a continuous Sobolev mapping defined on $\Heis^n$ and taking values in a metric space $Y$.  A standard definition of a $p$-Sobolev mapping, $1\leq p \leq \infty$, has two requirements: the mapping itself should be $p$-integrable, and the norm of the weak differential of the mapping should be $p$-integrable. As we consider only continuous mappings and our results are local in nature, only the second requirement is relevant.

There are many ways to define the class of continuous $p$-Sobolev mappings between metric spaces. We adopt the simplest and say that a continuous mapping $f \colon \Heis \to Y$ is in the \emph{local Sobolev class $\W^{1,p}_{\loc}(\Heis;Y)$} if there is an \emph{upper gradient} $g \colon \Heis^n \to [0,\infty]$ that is locally $p$-integrable with respect to $\cH^{2n+2}_\Heis$, i.e., a Borel function $g \in \Lloc(\Heis^n,\cH^{2n+2}_\Heis)$ such that for every rectifiable path $\gamma \colon [0,1] \to \Heis^n$,
$$d_Y(f(\gamma(0)),f(\gamma(1))) \leq \int_{\gamma} g \ ds,$$
where $ds$ refers to integration with respect to arclength. The upper gradient approach, developed by Cheeger \cite{Cheeger}, Heinonen and Koskela \cite{Acta}, and Shanmugalingam \cite{Nages}, is suitable in the general setting of doubling metric measure spaces that support a Poincar\'e inequality. For a thorough discussion of this and other possible approaches to the class of Sobolev mappings between metric spaces, see \cite{HKST}.

Aside from the definition, we shall only need one property of Sobolev mappings on $\Heis^n$. Namely, if $f \colon \Heis^n \to Y$ is a continuous mapping in the Sobolev class $\W^{1,p}_{\loc}(\Heis^n;Y)$ with $p>2n+2$, then \emph{Morrey's estimate} holds: there is a constant $c \geq 1$ and a dilation factor $\sigma \geq 1$, both depending only on $n$ and $p$, such that for any Heisenberg ball $B \subeq \Heis^n$,
\begin{equation}\label{Morrey}\diam f(B) \leq c(n,p)(\diam B)^{1-\frac{2n+2}{p}} \left(\int_{\sigma B} g^p \ d\cH^{2n+2}_\Heis\right)^{1/p}.\end{equation}
For a proof of \eqref{Morrey}, see \cite{SobMet} and \cite{HKST}. We note that Morrey's estimate in fact implies continuity; our a priori assumption of continuity is for convenience only.

\section{The proof of Theorem \ref{main}}\label{discussion section}


The kernel of the proof of Theorem \ref{main} can be traced back to Gehring's proof of the ACL property of metrically defined quasiconformal mappings \cite[Lemma~9]{rings}.  This argument was at first incorrectly applied to the Heisenberg setting by Mostow \cite{MostowBook}, who mistakenly asserted that the vertical projection map $\pi_{{\V}^\perp}$ is Lipschitz in the Heisenberg metric. This error was ingeniously overcome by Mostow \cite{MostowRemark}, and a simple presentation of the correct proof can be found in \cite{KR}.

Roughly speaking, adapting Gehring's original proof to the setting of Theorem \ref{main} will provide the claimed value of $\beta(p,m,\alpha)$ when $\alpha$ is large, and adapting Mostow's correct proof will provide the claimed value when $\alpha$ is small. There are significant obstacles to adapting Mostow's proof methods to the question of frequency of dimension distortion, as discussed in Section \ref{discussion section} below. The main tool in overcoming these obstacles is the slicing and projection machinery developed by Mattila \cite{mat:projections}, \cite{Mattila}.

\subsection{Gehring's method}\label{Gehring}
To motivate and organize our proof, we first give a brief outline of the proof of the fact that a quasiconformal homeomorphism  $f \colon \reals^3 \to \reals^3$ has the ACL property.

Using the coordinate system $(x,y,t)$ for $\reals^3$, for the moment we denote the $x$-axis by $V_x$, and set $\pi\colon \reals^3 \to V_x^\perp$ to be the standard Euclidean orthogonal projection.

We will show that for a closed ball $K$ containing the origin,
\begin{equation}\label{post comp line} \cH^{2}_{\reals^3}\left(\{a \in V_x^\perp \cap K: \cH^1_{\reals^3}(f(a+{V}_x))=\infty\}\right)=0.\end{equation}
This is not quite enough to show the ACL property, but provides sufficient intuition for our purposes.

We outline the three key steps of the proof that \eqref{post comp line} holds when $f$ is quasiconformal.

\begin{itemize}
\item[(i)]  Given $a=V_x^\perp \cap K$, if
$$
\liminf_{r \to 0}  \frac{\cH^3_{\reals^3}(f(\nbhd_{\reals^3}(a + {V}_x,r) \cap K))}{r^{2}}<\infty,$$
then $\cH^{1}_{\reals^3}(f((a+{V}_x) \cap K)) <\infty.$ This relationship between Minkowski content and Hausdorff measure  can be seen in this quasiconformal case by using the standard distortion estimate
$$\diam f(B) \simeq \left(\cH^{3}_{\reals^3}(f(B))\right)^{\frac{1}{3}},$$
which holds for any ball $B \subeq \reals^3$, along with a covering argument and H\"older's inequality.

\item[(ii)] Define a Radon measure $m$ on $V_x^\perp \cap K$ so that for each $a \in V_x^\perp \cap K$,
$$m(B_{V_x^\perp}(a,r)  \cap K) = \cH^3_{\reals^3}(f(\pi\inv(B_{V_x^\perp}(a,r)) \cap K)).$$
\item[(iii)] By the Radon-Nikodym Theorem, the derivative of this measure with respect to two-dimensional Hausdorff measure on $V_x^\perp \cap K$ exists and is finite $\cH^2_{V_x^\perp \cap K}$-almost everywhere. Thus
$$\lim_{r \to 0} \frac{ \cH^3_{\reals^3}(f(\pi\inv(B_{V_x^\perp}(a,r)) \cap K)))}{r^2} < \infty$$
for $\cH^2_{V_x^\perp \cap K}$-almost every point $a \in V_x^\perp \cap K$. Applying Step (i) now completes the proof, as in this setting we have
\begin{equation}\label{tube is preimage} \pi\inv(B_{V_x^\perp}(a,r)) = \nbhd_{\reals^3}(a + {V}_x,r).\end{equation}
\end{itemize}

\subsection{Adapting Gehring's method to the Heisenberg groups}\label{adapt Gehring}
We remind the reader of our convention that for a fixed horizontal subgroup $\V$, we denote
$$m:=\dim \V  = \dim_\reals \V = \dim_\Heis \V \ \text{and}\  w:=\dim_{\reals}\V^\perp=\dim_{\Heis} \V^\perp-1.$$
The goal of this section is to prove the following statement, which in particular gives the desired estimate in Theorem \ref{main} when $$\alpha \in \left[\frac{pm}{p-2},\frac{pm}{p-(w+1)}\right].$$

\begin{proposition}\label{big alpha}Let $Y$ be an arbitrary metric space and let $f \colon \Heis^n \to Y$ be a mapping in the Sobolev space $\W^{1,p}_{\loc}(\Heis;Y)$ for some $p > 2n+2$.  Given a horizontal subgroup $\V$ of $\Heis^n$ of dimension $1 \leq m \leq n$, and
$$\alpha \in \left[m,\frac{pm}{p-(w+1)}\right],$$
it holds that
$$\dim_{\reals}\{a \in \V^\perp:  \dim f(a*\V) \geq \alpha\} \leq (w+1)-p\left(1-\frac{m}{\alpha}\right).$$
\end{proposition}



For the remainder of this subsection, we assume the hypotheses of Proposition \ref{big alpha}.   In addition, we denote by $K$ the closure of an arbitrary bounded neighborhood of the origin in $\Heis^n$, and let $K'$ be the closure of a bounded neighborhood of the origin in $\Heis^n$ that contains $K$ in its interior.

As we wish to estimate the Euclidean dimension of a subset of $\V^\perp$, we will consider Euclidean balls in $\V^\perp$. For ease of notation, we define $W$ to be the metric space $(\V^\perp,d_{\reals^{2n+1}}),$ so that for $a \in \V^\perp$ and $r>0$
$$B_W(a,r) = \{a' \in \V^\perp: d_{\reals^{2n+1}}(a,a')<r\}.$$

The estimate in Proposition \ref{big alpha} is trivially true when $\alpha = m$, so we need only consider the case that $\alpha>m$. Define
$$E_\alpha = \{a \in \V^\perp: \cH^{\alpha}(f((a*\V) \cap K)) > 0 \}.$$
By basic properties of Hausdorff measure and dimension, it suffices to show
\begin{equation}\label{Eucl beta reduced}
\dim_{\reals} E_\alpha \leq (w + 1) - p\left(1-\frac{m}{\alpha}\right) \alpha. \end{equation}

By the countable additivity of Hausdorff measure, the following lemma allows us to assume without loss of generality that $E_\alpha$ is compact.

\begin{lemma}\label{union of compacts}
The set $E_\alpha$ is a countable union of compact sets.
\end{lemma}

\begin{proof}
As closed and bounded sets in $\reals^{2n+1}$ are compact, it suffices to show that $E_\alpha$ is a countable union of closed sets, which may then be decomposed into countably many closed and bounded parts. Since the $\alpha$-dimensional Hausdorff measure and the $\alpha$-dimensional Hausdorff content $\cH^\alpha_\infty$ have the same null sets, it suffices to show that for each $n \in \nats$, the set
$$E_\alpha(n)= \left\{a \in \V^\perp: \cH^{\alpha}_{\infty}(f((a*\V )\cap K))\geq \frac{1}{n}\right\}$$
is closed. Let $\{a_j\}_{j\in \nats} \subeq E_\alpha(n)$ be a sequence converging to a point $a \in \V^\perp$.  Since $f$ and $\pi_{\V^\perp}$ are continuous, for every $\ep>0$, there is an index $j(\ep) \in \nats$ such that if $j \geq j(\ep)$, then
$$f((a_j*\V) \cap K))\subeq \nbhd_{Y}(f((a*\V) \cap K)), \ep).$$
If $a \notin E_\alpha(n)$, then there is a cover $\{B_Y(y_i,r_i)\}_{i \in \nats}$ of $f((a*\V)\cap K)$ by open balls such that
$$\sum_{i \in \nats} r_i^\alpha < \frac{1}{n}.$$
Since $f((a*\V) \cap K)$ is compact, we may find $\ep>0$ such that the neighborhood $\nbhd_Y(f((a*\V) \cap K),\ep)$ is also covered by $\{B_Y(y_i,r_i)\}_{i \in \nats}$. This implies that
$$\cH^{\alpha}_{\infty}(f((a_j*\V)\cap K)) < \frac{1}{n}$$
for all $j \geq j(\ep)$, which yields the desired contradiction.
\end{proof}

We now establish a version of Step (i) in Gehring's method, which provides a sufficient condition for the desired bound on the dimension of the image of a line segment under $f$. It is only in the proof of this statement that we use the Morrey estimate.

\begin{proposition}\label{minkowski} Let $m \leq \alpha \leq p$. If
$$\liminf_{r \to 0}  \frac{\int_{\nbhd_{\Heis}(a * \V,r) \cap K'} g^p \ d\cH^{2n+2}_{\Heis}}{r^{(w+1)-p(1-\frac{m}{\alpha})}}<\infty,$$
then $\cH^{\alpha}(f(((a*\V) \cap K)) <\infty.$
\end{proposition}


\begin{proof}[Proof of Proposition \ref{minkowski}]
Fix $\ep>0$. By the Morrey estimate \eqref{Morrey}, there is a constant $\sigma \geq 1$ such that for $q \in \Heis^n$ and $r>0$, the Heisenberg ball $B_\Heis(q,r)$ satisfies
\begin{equation}\label{MorreyUse}
\diam f(B_\Heis(q,r)) \lesssim r^{1-\frac{2n+2}{p}} \left(\int_{B_\Heis(q,\sigma r)}g^p \ d\cH^{2n+2}_\Heis \right)^{1/p}.
\end{equation}
By the uniform continuity of $f$ on compact sets, we may find $\ep'>0$ so that if $B$ is a Heisenberg ball that intersects $K$ and has radius no greater than $\ep'$, then  $B \subeq K'$ and $\diam f(B) < \ep$  Let $r < \ep'$. Since $a * \V \subeq \Heis^n$ is isometric to $\reals^m$ equipped with the Euclidean metric, we may find a cover of $(a* \V)\cap K$ by Heisenberg balls $B_1,\hdots,B_N$ of radius $r/\sigma$ centered on $a*\V$, where $N \leq M'r^{-m}$ and $M'$ depends only on $M$ and $\sigma$. We may moreover assume that there is a number $D \geq 1$, depending only on $\sigma$, such that no point of the Heisenberg group lies in more than $D$ of the dilated balls $\sigma B_1,\hdots,\sigma B_N$.  Denoting by $\cH^{\alpha,\ep}$ the $\alpha$-dimensional Hausdorff pre-measure calculated by considering coverings by sets of diameter no greater than $\ep>0$, we see that
$$\cH^{\alpha,\ep}(f((a*\V)\cap K)) \leq \sum_{i=1}^N (\diam{f(B_i)})^\alpha.$$
Hence, by \eqref{MorreyUse},  H\"older's inequality, the bounded overlap property of the cover $\{\sigma B_i\}_{i=1}^N$, and the estimate $N \leq M'r^{-m}$ yield
\begin{align*}
\cH^{\alpha,\ep}(f((a*\V)\cap K)) \leq & r^{(1-((2n+2)/p))\alpha} \sum_{i=1}^N \left(\int_{\sigma B_i}g^p \ d\cH^{2n+2}_\Heis\right)^{\alpha/p} \\
\lesssim & r^{(1-((2n+2)/p))\alpha} N^{1-(\alpha /p)} \left(\int_{\nbhd_\Heis(a*\V,r) \cap K'} g^p \ d\cH^{2n+2}_\Heis \right)^{\alpha /p} \\
\lesssim & r^{(\alpha-m)-(w+1)\alpha/p} \left(\int_{\nbhd_\Heis(a*\V,r) \cap K'} g^p \ d\cH^{2n+2}_\Heis \right)^{\alpha /p}.
\end{align*}
The hypothesis implies that there is a number $c>0$, depending on $a$, such that if $r$ is sufficiently small, then
$$
\left(\int_{\nbhd_\Heis(a*\V,r) \cap K'} g^p \ d\cH^{2n+2}_\Heis \right)^{\alpha /p} \leq cr^{(w+1)\alpha/ p-(\alpha-m)}.
$$
Thus there is a quantity $c'>0$, independent of $\ep$, such that $\cH^{\alpha,\ep}(f(a * \V)) \leq c'$. Letting $\ep$ tend to zero yields the desired result.
\end{proof}

Now, we establish a version of Step (ii) in Gehring's method. We define a measure $\Phi$ on $\V^\perp$, depending on $g$ and $K$, by the following Carath\'{e}odory construction. For $\ep>0$ and $E \subeq \V^\perp$, set
$$
\Phi_\ep(E)= \inf\left\{\sum_{i \in \nats}\int_{\pi_{\V^\perp}\inv(B_{W}(a_i,r_i))\cap K'}g^p \ d\cH^{2n+2}_{\Heis}: E \subeq \bigcup_{i \in \nats} B_W(a_i,r_i), \ 0<r_i < \ep\right\}.$$
Then set
$$\Phi(E) = \lim_{\ep\to 0}\Phi_\ep(E).$$
Since $g^p \in L^1(\Heis^n)$, the hypotheses of \cite[Theorem 4.2]{Mattila} apply and the set function $\Phi$ defines a Borel regular measure on $W$.  It follows from the sub-additivity of the integral that given $a \in \V^\perp$ and $r>0$,
$$\Phi(B_W(a,r))=\int_{\pi_{\V^\perp}\inv(B_W(a,r)) \cap K'}g^p \ d\cH^{2n+2}_\Heis \ .$$
Hence $\Phi$ is a Radon measure on $W$. Note that if $f$ is a smooth diffeomorphism of $\Heis^n$ to itself, and $g$ is the norm of the differential of $f$, then $\Phi(B_W(a,r))$ is comparable to the $(2n+2)$-dimensional Hausdorff measure of the image $f(\pi_{\V^\perp}\inv(B_W(a,r)) \cap K)$, in analogy to the quasiconformal setting.

We now give the analog of Step (iii) in Gehring's method, thereby completing the proof of Proposition \ref{big alpha}.

\begin{proof}[Proof of Proposition \ref{big alpha}]
As mentioned above, it suffices to prove the estimate \eqref{Eucl beta reduced}. Suppose, by way of contradiction, that there exists a number $t$ such that
$$(w+1)-p\left(1-\frac{m}{\alpha}\right)< t < \dim_{W}E_\alpha.$$
Then $\cH^t(E_\alpha)=\infty$. By Lemma \eqref{union of compacts} we may reduce to the case that $E_\alpha$ is compact, and then find a compact subset $E \subeq E_\alpha$ such that $0<\cH^t(E)<\infty$ \cite[Theorem~8.19]{Mattila}. By Frostman's lemma \cite[Theorem~8.17]{Mattila}, there exists a nonzero Radon measure $m$ supported on $E$ with the property that $m(B_W(a,r)) \leq r^t$ for all $a \in W$ and all sufficiently small $r>0$.

By applying the Radon-Nikodym theorem \cite[Theorem 2.12]{Mattila} to $\Psi$ and $m$, we see that for $m$-almost every point $a \in W$,
\begin{equation}\label{mRN}
\lim_{r \to 0} \frac{\int_{\pi_{\V^\perp}\inv(B_W(a,r)) \cap K'}g^p \ d\cH^{2n+2}_\Heis}{m(B_W(a,r))} < \infty.
\end{equation}
Since $\pi_{\V^\perp} \colon \Heis^n \to \V^\perp$ is Lipschitz on compact sets when $\V^\perp$ is equipped with the Euclidean metric, there is some $L \geq 1$ such that
$$
\nbhd_\Heis(a*\V,r) \cap K' \subeq \pi_{\V^\perp}\inv(B_W(a,Lr))
$$
for all $a \in \V^\perp$. This fact, the Frostman estimate on $m$, and \eqref{mRN} imply that for $m$-almost every $a \in W$,
$$
\lim_{r \to 0} \frac{\int_{\nbhd_{\Heis}(a*\V,r) \cap K'} g^p \ d\cH^{2n+2}_{\Heis}}{r^t}<\infty.
$$
Since $t>\beta$, we may find a number $\alpha' < \alpha$ such that
$$t = (w+1)-p\left(1-\frac{m}{\alpha'}\right).$$
As we have assumed $t < \dim_{\reals}E_\alpha$, and clearly $\dim_{\reals}E_\alpha<w+1$, we may also assume that $\alpha'>m$.  Proposition \ref{minkowski} now implies that for $m$-almost every point $a \in E$, it holds that $\cH^{\alpha'}(f((a*\V) \cap K)) <\infty$.  However, since the non-zero measure $m$ is supported on $E$, which is a subset of $E_\alpha$, and $\alpha' < \alpha$, this is a contradiction. \end{proof}

\begin{remark}\label{metric statement} The argument given in this section generalizes to the metric space setting without substantial changes. We record this in the following statement and refer to \cite{DimDistMetric} for the relevant definitions.

\begin{theorem}\label{metric theorem} Let $n$ be a positive integer. Assume that $(X,d,\mu)$ is a proper metric measure space that is locally $Q$-homogeneous and supports a local $Q$-Poincar\'e inequality, $Q \geq n$. Let $\pi \colon X \to \reals^n$ be a Lipschitz map such that for each point $a \in \reals^n$, the preimage $\pi\inv(a)$ is locally $s$-homogeneous, $0 \leq s < Q$. If $p>Q$ and $f \colon X \to Y$ is a continuous mapping into a metric space with an upper gradient in $\Lloc^p(X)$, then for each $\alpha \in (s, ps/(p-Q+s)$,
$$\dim \{a \in \reals^n : \cH^\alpha(f(\pi\inv(a)))>0\} \leq (Q-s)-p\left(1-\frac{s}{\alpha}\right).$$
\end{theorem}
We note that Theorem \ref{metric theorem} implies Theorem \ref{foliation intro} in the case that the parameterizing space of the David--Semmes foliation is Euclidean.
\end{remark}

\subsection{Mostow's method}
We now outline how Mostow adjusted Gehring's method to show that a quasiconformal homeomorphism $f \colon \Heis \to \Heis$ has the ACL property.

Using the coordinate system $(x,y,t)$ for $\Heis$, for the moment we denote the $x$-axis by $\V_x$,  set $W_x$ to be the metric space $(\V_x^\perp, d_{\reals^2})$, which is isometric to the Euclidean plane, and denote by $\pi_{\V_x^\perp} \colon \V_x \to W_x$ the Heisenberg projection mapping defined by splitting $\Heis = \V_x^\perp \ltimes \V_x$; we emphasize here that the target $W_x$ is equipped with the Euclidean metric.

As null sets for $\cH^2_{\reals}\restrict W_x$ coincide with null sets for $\cH^3_{\Heis} \restrict \V_x^\perp$, we will show that for a closed ball $K$ in $\Heis$ containing the origin,
\begin{equation}\label{post comp line 2}
\cH^{2}_{\reals}\left(\{a \in W_x \cap K: \cH^1_{\Heis}(f(a*\V_x))=\infty\}\right)=0.\end{equation}
Again, this is not quite enough to show the ACL property, but provides sufficient intuition for our purposes.

We now describe three steps in the proof that $f$ has the ACL property.
\begin{itemize}
\item [$(\til{\rm{i}})$] The first step of Mostow's method is basically same as step (i) of Gehring's method, and Proposition~\ref{minkowski} has already accomplished its analog in the general setting of Theorem \ref{main}. Given $a\in W_x \cap K$, if
\begin{equation}\label{Heis Minkowski}
\liminf_{r \to 0}  \frac{\cH^4_{\Heis}(f(\nbhd_\Heis(a*{\V}_x,r) \cap K))}{r^{3}}<\infty,
\end{equation}
then $\cH^{1}(f((a*{\V}_x) \cap K)) <\infty.$ As before, this relationship between Minkowski content and Hausdorff measure can be seen in this quasiconformal case by using the standard distortion estimate
$$\diam f(B) \simeq \left(\cH^{4}_{\Heis}(f(B))\right)^{\frac{1}{4}},$$
which holds for any ball $B \subeq \Heis$, along with a covering argument and H\"older's inequality.

\item[$(\til{\rm{ii}})$] We now diverge from Gehring's method, as the denominator appearing in \eqref{Heis Minkowski} is $r^3$, and not $r^2$. We produce a measure, not on $W_x$ as in Gehring's method, but instead on the $y$-axis $W_{x,t} = \{(0,y,0):y \in \R\}$ inside of $W_x$. Let $\pi_{W_{x,t}} \colon W_x \to W_{x,t}$ denote the standard Euclidean orthogonal projection. We define a measure $m$  so that for each $y_0 \in W_{x,t}$
$$m(B_{W_{x,t}}(y_0,r) \cap K) = \cH^4_{\Heis}(f \circ \pi_{\V_x^\perp}\inv \circ \pi_{W_{x,t}}\inv(B_{W_{x,t}}(y_0,r) \cap K)).$$

\item[$(\til{\rm{iii}})$] As in Gehring's method, we may apply the Radon-Nikodym theorem to the measure $m$, but with respect to linear measure on $W_{x,t}$. Thus, for $\cH^1$-almost every $y_0$ in $W_{x,t}$,
\begin{equation}\label{measure Heis qc}\lim_{r \to 0} \frac{\cH^4_{\Heis}(f \circ \pi_{\V_x^\perp}\inv \circ \pi_{W_{x,t}}\inv(B_{W_{x,t}}(y_0,r) \cap K))}{r} < \infty.\end{equation}
We now claim that if \eqref{Heis Minkowski} fails to hold for a set $\mathcal{C} \subeq W_x \cap K$ of positive $\cH^2$-measure, then \eqref{measure Heis qc} will fail on $\pi_{W_{x,t}}(\mathcal{C})$, which has positive $\cH^1$-measure by Fubini's theorem, yielding a contradiction.
The key point in the proof of this claim is the relationship between the sets
$$\pi_{\V_x^\perp}\inv \circ \pi_{W_{x,t}}\inv(B_{W_{x,t}}(y_0,r) \ \text{and}\ \nbhd_{\Heis}((a*{\V}_x),r).$$
This relationship is clarified by a geometric statement: given $y_0 \in W_{x,t}$, points $a=(0,y_0,t)$ and $a'=(0,y_0,t') \in \pi_{W_{x,t}}\inv(y_0)$, and $r>0$, then
$$\nbhd_\Heis(a*{\V}_x,r) \cap \nbhd_\Heis(a_0*{\V}_x,r)$$
implies that
$$|t-t'| \lesssim r^2.$$
The claim now follows by a packing argument, completing the proof.
\end{itemize}

\subsection{Adapting Mostow's method to the Heisenberg groups}\label{small section}

We resume the notation of Section \ref{big alpha}, but now additionally assume that $\alpha \in \left(m,mp/(p-2)\right)$. We will show that
$$
\dim_{\reals} \{a \in \V^\perp: \cH^{\alpha}(f(a*\V \cap K))>0\} \leq w-\frac{p}{2}\left(1-\frac{m}{\alpha}\right).
$$
It suffices to show that for $m<\alpha'<\alpha$,
\begin{equation}\label{final goal} \dim_{\reals}\{a \in \V^\perp: \cH^{\alpha'}(f((a*\V) \cap K))=\infty\} \leq w-\frac{p}{2}\left(1-\frac{m}{\alpha'}\right).\end{equation}

As mentioned above, the analog of Step ($\til{i}$) in the preceding outline is accomplished in Proposition~\ref{minkowski}. Creating a measure as in Step ($\til{ii}$) is complicated by the following issue that arises in application of Fubini's theorem in Step ($\til{iii}$). For simplicity, we explain the complication only in the case $n=m=1$. The right hand side of \eqref{final goal} is less than $2$, and the projection $\pi_{W_{x,t}}$ maps some sets of dimension less than $2$ onto sets of zero $\cH^1_{W_{x,t}}$-measure, so no simple application of Fubini's theorem will suffice.  We overcome this problem by applying the projection and slicing machinery of Mattila \cite{Mattila} to conclude that for almost every co-dimension $1$ subspace of $W_x$, the corresponding projection of a set of dimension $t > 1$ has positive $\cH_{W_x}^{t-1}$-measure.

The following proposition, combined with Proposition \ref{minkowski}, quickly implies \eqref{final goal} and hence completes the proof of Theorem \ref{main}. Its proof, given Lemma \ref{tubes to slabs}, implements the strategy given in the previous paragraph.

\begin{proposition}\label{find big tubes} Suppose that $m < \alpha < mp/(p-2)$. Set
$$
\mathcal{C}_\alpha = \left\{a \in \V^\perp \cap K : \liminf_{r \to 0}  \frac{\int_{\nbhd_\Heis(a*\V,r) \cap K'} g^p \ d\cH^{2n+2}_\Heis}{r^{(w+1)-p(1-m/\alpha)}} =\infty\right\}.$$
Then
$$\dim_{\reals} \mathcal{C}_\alpha \leq w-\frac{p}{2}\left(1-\frac{m}{\alpha}\right).$$
\end{proposition}

Before proving Proposition \ref{find big tubes}, we establish some more notation related to the strategy outlined above. Denote
$$\mathbb{S}^{w-1}=\{\theta \in \V^\perp: ||\theta||_{\reals} = 1\}.$$
For $\theta \in \mathbb{S}^{w-1}$, denote by $\Theta$ the one-dimensional subspace of $\V^\perp$ generated by $\theta$, and by $\Theta^\perp$ its Euclidean orthogonal complement in $\V^\perp$. We denote the Euclidean orthogonal projection map by $\pi_{\Theta^\perp} \colon \V^\perp \to \Theta^\perp$, so that for $\hat{a} \in \Theta^\perp$,
$$\pi_{\Theta^\perp}\inv(\hat{a})=\{\hat{a}+\tau\theta: \tau \in \reals\} \subeq \V^\perp.$$
Note that if $\theta$ is parallel to the $t$-axis, then $\Theta^\perp = V^\perp \times \{0\}.$ In what follows we will consider only those $\theta$ that lie in a small neighborhood of the $t$-axis, meaning that $\Theta^\perp$ should be thought of as a small perturbation of $(V^\perp \times \{0\}) \subeq \V^\perp$.

We equip $\Theta^\perp$ with the restriction of the Euclidean metric on $\V^\perp$, so that for each $\hat{a} \in \Theta^\perp$ and $r>0$,
$$B_{\Theta^\perp}(\hat{a},r) = B_W(\hat{a},r) \cap \Theta^\perp.$$
We also associate the $(w-1)$-dimensional Hausdorff measure in the Euclidean metric to $\Theta^\perp$.

We use the projection and slicing machinery of \cite{Mattila} to conclude Proposition \ref{find big tubes} from the following lemma.

\begin{lemma}\label{tubes to slabs} Suppose that $m<\alpha < mp/(p-2).$ There is a subset $S \subeq \mathbb{S}^{w-1}$ with $\cH^{w-1}(S)>0$ such that if $\theta \in S$, then the following implication is true for $\cH^{w-1}$-almost every point $\hat{a} \in \Theta^\perp$. If $\mathcal{C} \subeq\pi_{\Theta^\perp}\inv(\hat{a}) \subeq \V^\perp$ has the property
\begin{itemize}
\item for every $k \in \nats$, there is a number $\ep(k) > 0$ such that for all $0<r<\ep(k)$ and $a \in \mathcal{C} \cap K$
$$
\int_{\nbhd_\Heis({a}*\V,r) \cap K'} g^p \ d\cH^{2n+2}_\Heis \geq k r^{(w+1)-p\left(1-\frac{m}{\alpha}\right)},$$
\end{itemize}
then
$\cH^{1-\frac{p}{2}\left(1-\frac{m}{\alpha}\right)}_{\reals}(\mathcal{C})=0$.
\end{lemma}

Assuming Lemma \ref{tubes to slabs} we complete the proof of Proposition \ref{find big tubes}.

\begin{proof}[Proof of Proposition \ref{find big tubes}]
Recall that
$$
\mathcal{C}_\alpha = \left\{a \in \V^\perp \cap K: \liminf_{r \to 0}  \frac{\int_{\nbhd_\Heis(a*\V,r) \cap K'} g^p \ d\cH^{2n+2}_\Heis}{r^{(w+1)-p(1-m/\alpha)}} =\infty\right\}.$$
As $\alpha < mp/(p-2)$,  it holds that
$$ w-\frac{p}{2}\left(1-\frac{m}{\alpha}\right) > w-1.$$
Fix any number $$\beta > w-\frac{p}{2}\left(1-\frac{m}{\alpha}\right),$$
and towards a contradiction assume that $\cH^\beta_{\reals}(\mathcal{C}_\alpha) >0$. One can check that $\mathcal{C}_\alpha$ is a Borel set, and so \cite[Theorem~8.13]{Mattila} implies that after passing to a subset, we may assume that $\cH^\beta_{\reals}(\mathcal{C}_\alpha)<\infty$ as well.

Let $k \in \nats$. For each $a \in \mathcal{C}_\alpha$, there exists a quantity $\ep(a,k)>0$ such that for each $0<r<\ep(a,k)$,
$$
\int_{\nbhd_\Heis(a*\V,r) \cap K'} g^p \ d\cH^{2n+2}_\Heis \geq kr^{(w+1)-p(1-m/\alpha)}$$
For $l \in \nats$, define
$$\mathcal{C}_{k,l} = \{a \in \mathcal{C}_\alpha: \ep(a,k) \geq 1/l\}.$$
As
$$
\mathcal{C}_{k,1} \subeq \mathcal{C}_{k,2} \subeq \hdots \subeq \bigcup_{l \in \nats} \mathcal{C}_{k,l} = \mathcal{C}_\alpha,
$$
we may choose natural numbers $l_k$ such that
$$\cH^\beta_{\reals}(\mathcal{C}_{k,l(k)}) > \left(1-2^{-(k+2)}\right) \cH^\beta_{\reals}(\mathcal{C}_\alpha).$$
It follows that the set
$$\mathcal{C}=\bigcap_{k \in \nats} \mathcal{C}_{k,l(k)}$$
also satisfies $0<\cH^\beta_{\reals}(\mathcal{C})<\infty.$
Moreover, for each index $k \in \nats$, radius $0<r \leq l(k)\inv$, and point $a \in \mathcal{C}$, it holds that
$$
\int_{\nbhd_\Heis(a*\V,r) \cap K'} g^p \ d\cH^{2n+2}_\Heis \geq  kr^{(w+1)-p(1-m/\alpha)}.
$$
Let $S$ be the positive measure subset of $\mathbb{S}^{w-1}$ guaranteed by Lemma \ref{tubes to slabs}.  If $\theta \in S$, then for $\cH^{w-1}$-almost every $\hat{a} \in \Theta^\perp$,
\begin{equation}\label{zero measure slice}
\cH^{1-\frac{p}{2}\left(1-\frac{m}{\alpha}\right)}_{\reals}(\pi_{\Theta^\perp}\inv(\hat{a}) \cap \mathcal{C})=0.
\end{equation}

However, because $\beta >w-1$, it follows from \cite[Theorem 8.9 and Corollary 9.8]{Mattila} that for $\cH^{w-1}$-almost every $\theta \in \mathbb{S}^{w-1}$ the Euclidean orthogonal projection $\pi_{\Theta^\perp}(\mathcal{C}) \subeq \Theta^\perp$ has positive (and finite) $\cH^{w-1}$-measure.  Hence, by \cite[Theorem 10.10]{Mattila}, for $\cH^{w-1}$-almost every $\theta \in \mathbb{S}^{w-1}$, there is a set $A \in\Theta^\perp$ of positive $\cH^{w-1}$-measure such that if $\hat{a} \in A$, then
\begin{equation}\label{key estimate}
\dim_{\reals} \pi_{\Theta^\perp}\inv(\hat{a}) \cap \mathcal{C} = \beta-(w-1)> 1-\frac{p}{2}\left(1-\frac{m}{\alpha}\right).
\end{equation}
In particular, we may choose $\theta \in S$ and $\hat{a} \in \Theta^\perp$ such that \eqref{key estimate} and \eqref{zero measure slice} hold, which is a contradiction.
\end{proof}

The proof of Lemma \ref{tubes to slabs} roughly corresponds to Step ($\til{ii}$) in Mostow's method, with the slight perturbation of $\V^\perp$ taken into account. Given $\theta \in \mathbb{S}^{w-1}$,  we define an appropriate measure on $\Theta^\perp$.  For $\hat{a} \in \Theta^\perp$ and $r>0$, we define a ``tilted slab" in $\Heis^n$ by
$$\Sl(\hat{a},r) = \pi_{\V^\perp} \inv \circ \pi_{\Theta^\perp}\inv (B_{\Theta^\perp}(\hat{a},r)).$$
We define a measure $\Psi$ on $\Theta^\perp$ by the following Carath\'{e}odory construction. For $\ep>0$ and $E \subeq \Theta^\perp$, set
$$\Psi_\ep(E)= \inf\left\{\sum_{i \in \nats}\int_{\Sl(\hat{a}_i,r)) \cap K}g^p \ d\cH^{2n+2}_{\Heis}: E \subeq \bigcup_{i \in \nats} B_{\Theta^\perp}(\hat{a}_i,r_i), \hat{a}_i \in\Theta^\perp,\ 0<r_i < \ep\right\}.$$
Then set
$$\Psi(E) = \lim_{\ep\to 0}\Psi_\ep(E).$$
Since $g^p \in \Lloc^1(\Heis^{n})$, the hypotheses of \cite[Theorem 4.2]{Mattila} apply, and so the set function $\Psi$ defines a Borel regular measure on $\Theta^\perp$.  It follows from the sub-additivity of the integral that given $\hat{a} \in \Theta^\perp$ and $r>0$,
$$\Psi(B_{\Theta^\perp}(\hat{a},r))=\int_{\Sl(\hat{a},r) \cap K}g^p \ d\cH^{2n+2}.$$
Hence $\Psi$ is a Radon measure on $\Theta^\perp$.

By applying the Radon-Nikodym theorem \cite[Theorem 2.12]{Mattila} to $\Psi$ and $\cH^{w-1}$, we see that for $\cH^{w-1}$-almost every point $\hat{a}\in \Theta^\perp$,
\begin{equation}\label{RN}
\lim_{r \to 0} \frac{\int_{\Sl(\hat{a},r)) \cap K}g^p \ d\cH^{2n+2}_\Heis}{r^{w-1}} < \infty.
\end{equation}

The following geometric lemma is the key point of the proof of Lemma~\ref{tubes to slabs}; it corresponds to Step ($\til{iii}$) in Mostow's method.

\begin{lemma}\label{disjoint tubes} There is a set $S \subeq \mathbb{S}^{w-1}$ of positive $\cH^{w-1}$-measure, depending only on $K$, with the following property. Given $\theta \in S$, $\hat{a} \in \Theta^\perp$, and
\begin{align*} a_1&=\hat{a}+t_1\theta \in \pi_{\Theta^\perp}\inv(\hat{a}) \subeq \V^\perp\\
					 a_2 &=\hat{a}+t_2\theta  \in \pi_{\Theta^\perp}\inv(\hat{a}) \subeq \V^\perp,\end{align*}
where $t_1,t_2 \in \reals$, if there is $r>0$ such that
\begin{equation}\label{tube overlap}
\nbhd_\Heis({a_1}*\V,r) \cap \nbhd_\Heis({a_2}*\V,r) \cap K \neq \emptyset,
\end{equation}
then $|t_1 - t_2| \leq 8r^2$.
\end{lemma}

\begin{proof}
Assuming that \eqref{tube overlap} holds, we may find elements $v_1$ and $v_2$ in $\V$ such that
\begin{equation}\label{norm bound}
||(a_1*v_1)\inv*(a_2*v_2)||_{\Heis} < 2r.
\end{equation}
We denote the Euclidean orthogonal projection of $\reals^{2n+1}$ onto $\reals^{2n}$ by $\pi_{\reals^{2n}}$, and the Euclidean orthogonal projection  of $\reals^{2n+1}$ onto the $t$-axis (i.e., the last coordinate of $\reals^{2n+1}$) by $\pi_t$. For ease of notation, we omit reference to $\pi_{\reals^{2n}}$ in the arguments of the symplectic form $\omega$, so that for points $a,a' \in \Heis^n$, we write
$$\omega(a,a'):=\omega(\pi_{\reals^{2n}}(a),\pi_{\reals^{2n}}(a')).$$

We note first that by the linearity of $\pi_{\reals^{2n}}$ and the fact that $\omega$ is bi-linear and anti-symmetric,
\begin{align*}  a_1\inv*a_2 &= \left((t_2-t_1)\pi_{\reals^{2n}}(\theta),(t_2-t_1)\pi_{t}(\theta)+2\omega(\hat{a} + t_2\theta,\hat{a} + t_1\theta) \right)\\
&=(t_2-t_1) \left(\pi_{\reals^{2n}}(\theta),\pi_{t}(\theta)+2\omega(\theta,\hat{a}) \right).\end{align*}
Define $\tau$ to be the $t$-component of $(a_1*v_1)\inv*(a_2*v_2)$. Using the above equation and  the fact that both $\omega$ and $\pi_t$ vanish on $\V$, we now compute that
$$\tau = (t_2-t_1)\left(\pi_t(\theta)+2\omega(\theta,\hat{a}+v_1+v_2)\right)$$

As $\hat{a}$, $v_1$, and $v_2$ may all be assumed to lie in a fixed compact set depending only on $K$, whenever $\theta$ is in a sufficiently small neighborhood $S \subeq \mathbb{S}^{w-1}$ of the unit vector in the $t$-direction,
$$|\pi_t(\theta)+2\omega(\theta,\hat{a}+v_1+v_2)| \geq |\pi_t(\theta)| -|2\omega(\theta,\hat{a}+v_1+v_2)| \geq \frac{1}{2}$$
and hence $2|\tau| \geq |t_2-t_1|.$
The definition of the Kor\'anyi norm on $\Heis^n$ and \eqref{norm bound} now yield
$$4r^2 \geq \frac{|t_2-t_1|}{2},$$
as desired. \end{proof}

We now provide the proof of Lemma \ref{tubes to slabs}, and so complete the proof of Theorem~\ref{main}.

\begin{proof}[Proof of Lemma \ref{tubes to slabs}]
Let $S \subeq \mathbb{S}^{w-1}$ be the set provided by Lemma \ref{disjoint tubes}, and let $\theta \in S$. Recall that $\cH^{w-1}$-almost every point $\hat{a} \in \Theta^\perp$ satisfies the Radon-Nikodym estimate \eqref{RN}. Let $\hat{a}$ be such a point, and suppose that $\mathcal{C} \subeq \pi_{\Theta^\perp}\inv(\hat{a})\subeq \V^\perp$ has the property that
for every $k \in \nats$, there is a number $\ep(k) > 0$ such that
\begin{equation}\label{tube assumption}
\int_{\nbhd_\Heis(a*\V,r) \cap K'} g^p \ d\cH^{2n+2}_\Heis \geq k r^{(w+1)-p\left(1-\frac{m}{\alpha}\right)}
\end{equation}
for all $0<r<\ep(k)$ and $a \in \mathcal{C}$. Working towards a contradiction, we assume that $\cH_\reals^\gamma(\mathcal{C})>0$ where
$$
\gamma = 1- \frac{p}{2}\left(1-\frac{m}{\alpha}\right).
$$

Since $\pi_{\Theta^\perp}\circ \pi_{\V^\perp} \colon \Heis^n \to \Theta^\perp$ is Lipschitz on compact sets (recall that we have equipped $\Theta^\perp$ with the Euclidean metric), there is a number $\kappa \geq 1$, depending only on $K$ such that
$$
\nbhd_\Heis(a*\V, r/\kappa) \cap K  \subeq \Sl(\hat{a},r)
$$
for all $a \in \pi_{\Theta^\perp}\inv(\hat{a})$ provided $r>0$ is sufficiently small.

Consider a maximal $8(r/\kappa)^2$-separated set $\{a_i\}_{i=1}^{N_r} \subeq \mathcal{C} \cap K$; as usual we use the Euclidean metric on $\mathcal{C}\subeq \V^\perp$. Lemma~\ref{disjoint tubes} implies that the corresponding family $\{\nbhd_\Heis(a_i*\V,r/\kappa)\cap K\}_{i=1}^{N_r}$ is disjoint. Hence, the above statements and \eqref{tube assumption} imply that for sufficiently small $r>0$,
\begin{align*}
N_{r}k\left(\frac{r}{\kappa}\right)^{(w+1)-p\left(1-\frac{m}{\alpha}\right)} &\leq \sum_{i=1}^{N_r} \int_{\nbhd(a_i*\V,r/\kappa)\cap K} g^p \ d\cH^{2n+2}_\Heis \\ & \leq   \int_{\Sl(\hat{a},r) \cap K} g^p \ d\cH^{2n+2}_\Heis \, .
\end{align*}
Moreover, since $\cH_{\reals}^\gamma(\mathcal{C})>0,$ it holds that
$$\liminf_{r \to 0} N_r \gtrsim r^{-2\gamma}.$$
Combining these estimates with the definition of $\gamma$ shows that
$$k= k \liminf_{r \to 0} r^{-2\gamma+ (w+1)-p\left(1-\frac{m}{\alpha}\right) -(w-1)}  \lesssim \liminf_{r \to 0}\frac{\int_{\Sl(\hat{a},r) \cap K} g^p \ d\cH^{2n+2}_\Heis}{r^{w-1}} .$$
Letting $k$ tend to infinity contradicts \eqref{RN}, and yields the desired result.
\end{proof}

This line of reasoning establishes \eqref{final goal} and consequently completes the proof of Theorem \ref{main}.

\section{A mapping that increases the dimension of many lines} \label{4 corners}
We now prove Theorem \ref{four corners intro}. The construction is similar in spirit to those given in \cite[Theorem~1.3]{DimDistMetric} and \cite[Section~4]{BMT}.

\begin{proof}[Proof of Theorem \ref{four corners intro}]  We consider the foliation of $\Heis$ by left translates of the horizontal subgroup $\V$ defined by the $x$-axis; the same construction works for any horizontal subgroup of $\Heis$. Again, we set $W=(\V^\perp,d_{\reals^2})$, and we define for $a \in \V^\perp$ and $s \in \reals$
$$
a(s) = a*(s,0,0)$$
Let $p>4$ and let $\alpha \in [1,\frac{p}{p-2}]$. By the Dimension  Comparison Theorem, it suffices to show that there is a compact set $E \subeq \V^\perp$ and a continuous mapping $f \colon \Heis \to \reals^2$ with an upper gradient in $\rm{L}^p(\Heis)$ such that
$$
\dim_{\reals} E = 2- p\left(1-\frac{1}{\alpha}\right),
$$
and $\dim f(a*\V) = \alpha$ for every $a \in E$.

Let $$\beta = 2- p\left(1-\frac{1}{\alpha}\right)$$ and choose $0<\sigma<1$ such that $$4\sigma^{\beta} =1.$$
We consider the iterated function system defined by the (Euclidean) similarities $f_i \colon W \to W$, $i=1,\hdots,4$, where
\begin{align*} f_1((0,y,t))&=(0,\sigma y,{\sigma}t),\\
					 f_2((0,y,t))&=(0,{\sigma}y,{\sigma}t)+(0,1-{\sigma},0),\\
					 f_3((0,y,t))&=(0,{\sigma}y,{\sigma}t)+(0,0,1-{\sigma}),\\
					 f_4((0,y,t))&=(0,{\sigma}y,{\sigma}t)+(0,1-{\sigma},1-{\sigma}).
\end{align*}
The unique compact invariant set $F^\alpha$ of this system is also known as a {\it four-corner set} or {\it Garnett set}. The set $F^\alpha$ can be expressed explicitly in the following way. Let $I=\{0\} \times [0,1] \times [0,1]$ be the (Euclidean) unit square in $W$.  For $m \in \nats$, let $\mathcal{S}_m$ denote the sequences $\omega = (\omega_1,\hdots,\omega_n)$ of length $m$ with entries in the set $\{1,\hdots,4\}$. We employ the convention that $\mathcal{S}_0$ contains the empty sequence. For $\omega \in \mathcal{S}_m$, define
$$f_\omega = f_{\omega_1} \circ \hdots \circ f_{\omega_m}.$$
Set $F^\alpha_\omega = f_\omega(I).$ Then
$$F^\alpha = \bigcap_{m\in \nats} \bigcup_{\omega \in \mathcal{S}_m} F^\alpha_\omega.$$
The iterated function system satisfies the open set condition, and so
$$0< \cH_{\reals}^{\beta}(F^\alpha)<\infty$$
by Hutchinson's Theorem \cite{Hutchinson}.

Consider a diffeomorphism $\phi \colon \reals^3 \to \reals^3$ with the property that $\pi_W = P_W \circ \phi$, where $P_W(x,y,t) = (0,y,t)$ is the standard Euclidean orthogonal projection onto $W$. Fix $n \in \nats$ and $\omega \in \mathcal{S}_n$. We define a ``column" $\mathcal{C}_\omega$ over $F^\alpha_\omega$ by
$$\mathcal{C}_\omega = \phi\inv(\{(x,y,t): (0,y,t) \in F^\alpha_\omega\ \text{and}\ x\in [0,1]\}).$$
Thus, if $a \in F^\alpha_\omega$, then the point $a(s)$ of the leaf $a*\V$ is in $ \mathcal{C}_\omega$ for any $s \in [0,1]$.

Let $X_\omega$ be a maximal $\sigma^m$-separated set (in the Heisenberg distance) in $\mathcal{C}_\omega$. By volume considerations,  we see that
$$\card X_\omega \lesssim \sigma^{-2m}.$$
Denote
$$\mathcal{Q}_m =\left\{B_\Heis(z,\sigma^m): z \in \bigcup_{\omega \in \mathcal{S}_m}X_\omega \right\}
$$
and
$$
\mathcal{Q} = \bigcup_{m \in \nats}{\mathcal{Q}_m}.
$$
There exists a quantity $C\geq 1$, depending only on $\alpha$, such that if $\omega$ and $\omega'$ are sequences in $\mathcal{S}_n$, then
$$d_\Heis(\mathcal{C}_{\omega},\mathcal{C}_{\omega'}) \geq d_{\reals}(\mathcal{C}_{\omega},\mathcal{C}_{\omega'}) \geq \frac{\sigma^m}{C}.$$
Hence, for some possibly larger quantity $C \geq 1$, also depending only on $\alpha$,
the collection $\{(1/C)B: B \in \mathcal{Q}_m\}$ is disjoint. Since $\Heis$ is Ahlfors regular, we conclude that
\begin{equation}\label{bounded overlap as}\sup_{z \in \Heis} \sum_{B \in \mathcal{Q}_m} \chi_{100B}(z) <\infty.\end{equation}
For each $B \in \mathcal{Q}$, we may find a Lipschitz function $\psi_B \colon \Heis \to [0,1]$ such that $\psi_B|_{\ovl{B}} =1$, the support of $\psi_B$ is contained in $2B$, and
$$\Lip \psi_B \lesssim (\diam B)\inv.$$
Let $\xi\colon \{B \in \mathcal{Q}\} \to \ovl{B}_{\reals^N}(0,1)$ be measurable. For each $m \in \nats$, define $f_{\xi,m} \colon X \to \reals^N$ by
$$f_{\xi,m}(x) = \sum_{B \in \mathcal{Q}_m} \sigma^{\frac{m}{\alpha}}\psi_B(x)\xi(B).$$
Finally, define $f_\xi \colon X \to \reals^N$ by
$$f(x) = \sum_{m \in \nats}(1+m)^{-2}f_{\xi,m}(x).$$
Then $f_\xi$ is continuous and bounded. We claim that $\Lip f_\xi$ is an upper gradient of $f_\xi$, and that $\Lip f_{\xi}\in\rm{L}^p(\Heis)$. It suffices to show, for each $n \in \nats$, that $\Lip f_{\xi,m}$ is an upper gradient of $f_{\xi,m}$, and that the norms $\{||\Lip f_{\xi,m}||_{\rm{L}^p}\}_{m \in \nats}$ are uniformly bounded.  The first statement follows from the fact that $f_{\xi,m}$ is locally Lipschitz \cite{LAMS}. For the second fact, we calculate, using the bounded overlap
condition \eqref{bounded overlap as}, that
\begin{align*}
||\Lip f_{\xi,m}||^p_{{\rm{L}}^p} & \lesssim \sum_{\omega \in \mathcal{S}_m} \sum_{B \in X_\omega} \int_{2B} (\diam B)^{-p}\sigma^{\frac{mp}{\alpha}} \ d\cH^4_\Heis  \\ & \lesssim  4^m \sigma^{-2m} \sigma^{-mp} \sigma^{\frac{mp}{\alpha}} \sigma^{4m}=(4\sigma^{\beta})^m = 1.
\end{align*}

We now consider the vectors $\xi_B$ to be chosen randomly, i.e., we assume that the functions $\{\xi_B\}_{B \in \mathcal{Q}}$ are independent random variables distributed according to the uniform probability distribution on the closed unit ball $\ovl{B}_{\reals^N}(0,1)$. This makes the function $\xi$ into a random variable on the same probability space.

We will prove that for every $\alpha' < \alpha$,
$$\E_\xi \left( \int_{F_\alpha} I_{\alpha'} ( (f_\xi)_{\#}(\cH^{1}\restrict a*\V) \ d\cH^\beta(a) \right) < \infty,
$$
which implies that almost surely in $\xi$, it holds that $\dim f_\xi(a*V) \ge \alpha'$ for $\cH^\beta$-almost every $a \in F_\alpha$.  Thus, almost surely in $\xi$, the full-measure set of points $a \in F_\alpha$ where this occurs satisfies the requirements on the set $E$ in the statement of the theorem.

By the Fubini--Tonelli theorem, the definition of the energy functional, and \cite[Theorem~1.19]{Mattila}, it suffices to show that
$$
\int_{[0,1]} \int_{[0,1]}\int_{F_\alpha} \E_\xi \left( |f_\xi(a(s)) - f_\xi(a(s'))|^{-\alpha'} \right) \, d\cH^\beta(a)\, d\cH^1(s) \, d\cH^1(s') < \infty.
$$
For $a \in F$ and $s, s' \in [0,1]$, we write
$$
f_\xi(a(s)) - f_\xi(a(s')) = \sum_{B \in \mathcal{Q}} c_B(a,s,s') \, \xi_B,
$$
where for $B \in \mathcal{Q}_m$
$$
c_B(a,s,s') = (1+m)^{-2} \, \sigma^{\frac{m}{\alpha}}( \psi_B(a(s)) - \psi_B(a(s')) )
$$
We denote by $c(a,s,s') $ the supremum of the set of numbers $\{c_B(a,s,s')\}_{B \in \mathcal{Q}}.$
Note that $c(a,s,s')= |c_{B}(a,s,s')|$ for some $B \in \mathcal{Q}$, as
$$\sum_{B \in \mathcal{Q}}|c_B(s,s')| < \infty.$$

By \cite[Lemma 4.4]{BMT}, since $\alpha' < \alpha < N$ it holds that
$$\E_\xi \left( |f_\xi(a(s)) - f_\xi(a(s'))|^{-\alpha'} \right) \lesssim c(a,s,s')^{-\alpha'} \, .$$
In view of this, it remains to show that
$$
\int_{[0,	1]}\int_{[0,1]}\int_{F} c(a,s,s')^{-\alpha'} \, d\cH^\beta(a) d\cH^1(s) \, d\cH^1(s') < \infty.
$$
We will in fact show the stronger statement
$$
\sup_{a \in F}\sup_{s \in [0,1]} \int_{[0,1]} c(a,s,s')^{-\alpha'} \, d\cH^1(s') <\infty.
$$
Fix $a \in F$ and $s \in [0,1]$.  For each $s' \in [0,1]$, define $m(s') \in \nats$ by
$$2^{-m(s')+2} \leq d_{\Heis}(a(s),a(s')) < 2^{-m(s')+3}.$$
Find $B \in \mathcal{Q}_{m(s')}$ that contains $a(s)$. Then $a(s') \in 100B \bslash 2B$, and so
$$
c(a,s,s')||\ge |c_{B}(a,s,s')| = (1+n(s'))^{-2} \sigma^{\frac{n(s')}{\alpha}}.
$$
For each $m \in \nats$, denote by $E_m$ the set of points $s' \in [0,1]$ for which $m(s')=m$, and let $B_m \in \mathcal{Q}_{m}$ contain $a(s)$. By the above argument, $\cH^1(E_m) \leq \cH^1(100B_m \cap a(s)) \lesssim \sigma^m$. Hence
\begin{align*}
\int_{[0,1]}c(a,s,s')^{-\alpha'} \, d\cH^1(s) &= \sum_{m \in \nats} \int_{E_m} c(a,s,s')^{-\alpha'} \, d\cH^1(s) \\
& \le \sum_{m \in \nats} m^{2\alpha'} \sigma^{\frac{-m\alpha'}{\alpha}} \cH^1(E_m)\\
& \lesssim \sum_{m \in \nats} m^{2\alpha'}  \sigma^{m\left(1-\tfrac{\alpha'}{\alpha}\right)}
\end{align*}
Since $\alpha'<\alpha$, the final sum above converges to a value independent of $a \in F$ and $s \in [0,1]$. This completes the proof.
\end{proof}

\begin{remark}
It seems likely a mapping as in Theorem \ref{four corners intro} can be found for \emph{many} sets $E \subeq W$ of dimension $2- p\left(1-\frac{1}{\alpha}\right)$; the key property is that the set $E$ should be \emph{evenly coverable}, i.e., there exist constants $\sigma, C \geq 1$ such that for all sufficiently small $\ep>0$, there is a cover $\{B(x_k,r_k)\}_{k \in \nats}$ of $E$ by balls centered in $E$ such that
\begin{itemize}
\item[i)]$\sup_{k \in \nats} r_k < \ep$,
\item[ii)] $\sum_{k \in \nats} r_k^{\dim E} <C$
\item[iii)] $\sup_{x \in X} \sum_{k \in \nats} \chi_{B(x_k,\sigma r_k)}(x) < C.$
\end{itemize}
For further discussion of the notion of even coverability, see \cite[Section 7]{DimDistMetric}.
\end{remark}

\section{Questions}\label{questions section}

We conclude this paper with several questions motivated by the results.

\begin{question}\label{recover}
Can the conclusion of Theorem \ref{main} be modified to read
\begin{equation}\label{measure 0 question}  \cH_{\reals}^\beta\left(\{a \in \V^\perp:  \dim f(a*\V) > \alpha\}\right)=0\end{equation}
for $\alpha>m$?

The techniques of this paper do not seem to be sufficient to prove \eqref{measure 0 question}, although it seems likely that, after minor modifications, one could prove that for each compact neighborhood $K$ of the origin in $\Heis^n$,
$$ \cH_{\reals}^\beta\left(\{a \in \V^\perp:  \cH^\alpha (f(a*\V) \cap K) = \infty\}\right)=0.$$
\end{question}

\begin{question}\label{example}
The example in Theorem \ref{four corners intro} is only obtained for small $\alpha$, namely, $\alpha<\tfrac{p}{p-2}$. Denote by $\V_x$ the $x$-axis in $\Heis$, and let $p>4$. Does there exist a decreasing bijection $\beta \colon \left(1,\frac{p}{p-3}\right) \to (0,2)$, such that for each $\alpha \in \left(1,\frac{p}{p-3}\right)$, there is a compact set $E \subeq \V_x^\perp$ and a continuous mapping $f \in \W^{1,p}(\Heis;\reals^4)$ satisfying $0<\cH_{\reals^3}^{\beta(\alpha)}(E) < \infty$ and $\dim f(a*\V) \geq \alpha$ for every $a \in E$?
\end{question}

\begin{question}\label{ideal}
As mentioned in the introduction, our results combined with the Dimension Comparison Theorem give an estimate on the Heisenberg dimension of the set of left cosets whose dimensions are badly distorted by a Sobolev mapping.  However, they are likely not best possible. Does an upper bound of the form
$$\dim_{\Heis} \{a \in \V^\perp: \dim f(a*\V) \geq \alpha\} \leq (2n+2-m)-p\left(1-\frac{m}{\alpha}\right)$$
hold? A positive answer would indicate that, with respect to this problem, the Heisenberg group has the same behavior as $\reals^{2n+2}$.
\end{question}

\bibliographystyle{acm}
\bibliography{HDimDist}
\end{document}